\documentclass{amsart}
\usepackage{amssymb,latexsym}

\usepackage{xcolor}

\usepackage[all,arc,curve,color,frame,pdf,line]{xy}

\newtheorem{thrm}{Theorem}[section]
\newtheorem{lem}[thrm]{Lemma}

\newtheorem{cor}[thrm]{Corollary}
\newtheorem{notation}[thrm]{Notation}
\newtheorem{problem}[thrm]{Problem}

\theoremstyle{definition}
\newtheorem{definition}[thrm]{Definition}
\newtheorem{example}{Example}[section]

\theoremstyle{remark}

\numberwithin{equation}{section}

\def\Z{{\mathbb Z}}

\def\la{{\langle}}
\def\ra{{\rangle}}
\def\tl{{\unlhd}}
\def\fix{{\rm fix}}

\def\AGL{{\rm AGL}}
\def\Aut{{\rm Aut}}
\def\Stab{{\rm Stab}}

\def\cal{\mathcal}
\def\PStab{{\rm PStab}}

\def\C{{\rm Z}}
\def\clo{{\rm clo}}

\def\WStab{{\rm WStab}}
\def\Cay{{\rm Cay}}
\def\Cos{{\rm Cos}}

\begin{document}
\title[]{New classes of groups related to algebraic combinatorics with applications to isomorphism problems}

\author{Ted Dobson}
\address{University of Primorska, UP IAM, Muzejski trg 2, SI-6000 Koper, Slovenia and
\\University of Primorska, UP FAMNIT, Glagolja\v{s}ka 8, SI-6000 Koper, Slovenia}
\email{ted.dobson@upr.si}

\begin{abstract}
\smallskip
We introduce two refinements of the class of $5/2$-groups, inspired by the classes of automorphism groups of configurations and automorphism groups of unit circulant digraphs.  We show that both of these classes have the property that any two regular cyclic subgroups of a group $G$ in either of these classes are conjugate in $G$.  This generalizes two results in the literature (and simplifies their proofs) that show that symmetric configurations and unit circulant digraphs are isomorphic if and only if they are isomorphic by a group automorphism of $\Z_n$.

\end{abstract}

\maketitle

In \cite{Dobson2023apreprint}, a class of groups called $5/2$-closed groups was defined which properly contains all automorphism groups of vertex-transitive digraphs.  A sufficient condition for the quotient of a $5/2$-closed group to be $5/2$-closed was given, and this was used to determine all Sylow $p$-subgroups of $5/2$-closed groups of odd prime-power degree that contain a regular cyclic subgroup.  Together with work in \cite{Dobson2023epreprint}, this gives the automorphism group of all circulant digraphs of odd prime-power order (Cayley digraphs of $\Z_{p^k}$, $p$ odd and $k\ge 1$).

In this paper, we continue this program and introduce two additional families of groups.  We will call these families $3/2$-closed and $9/8$-closed groups.  One should think of $9/8$-closed groups as being the analogues of the automorphism groups of configurations, in the way that $5/2$-closed groups were inspired by the automorphism groups of vertex-transitive digraphs.  Continuing the analogies, $3/2$-closed groups are thought of as analogues of the automorphism groups of unit circulant digraphs.  We go ahead and define an intermediary class of groups, $5/4$-closed groups, but at this time it is not clear if they will be of general interest.

We will show that for $n\ge 1$, any two regular cyclic subgroups in either a $3/2$-closed group $G$ (Theorem \ref{Toida generalization}) or a $9/8$-closed group $G$ (Theorem \ref{configuration result}) are conjugate in $G$.  These results generalize several known results.  First, Toida conjectured that every unit circulant digraph is a CI-digraph.  This conjecture was independently verified to be true by two groups of authors, the author and Joy Morris \cite{DobsonM2002} using permutation group techniques, as well as by Muzychuk, Klin and P\"oschel \cite{MuzychukKP2001} using the method of Schur.  The more general result proven here has a shorter proof.  Koike, Kov\'acs, Maru\v si\v c, and Muzychuk showed that $\Z_n$ is a CI-group with respect to balanced configurations \cite{KoikeKMM2019}.  We prove this in a much more general situation, as well as with a shorter proof.

Additionally, in Theorem \ref{compute G^5/2} we will also give a method of construction of the $5/2$-closure of any transitive permutation group $G$.  This construction illustrates that the $5/2$-closure is determined internally (i.e. all of the information one needs to construct it is already contained in the group).  The $2$-closure, by contrast, can be thought of as ``external".  To construct it, one calculates the orbital digraphs of $G$, then calculates the automorphism group of each orbital digraphs, and then the $2$-closure is the intersection of all such automorphism groups.  Calculating the automorphism groups of the orbital digraphs is generally difficult, and the relationship between the automorphism group and the group $G$ may not be clear.

Finally, there are in the author's view several other interesting results in this paper.  In Theorem \ref{9/8 object} we show an incidence structure where each pair of points is on at most one line and each line contains at least three points (a partial Sylvester-Gallai design) has $9/8$-closed automorphism group (as well as for such objects that have been ``colored" and ``directed"), and that every connected vertex-transitive digraph of girth at least $5$ also has $9/8$-closed automorphism group (Theorem \ref{graph 9/8 closed}).  In Theorem \ref{3/2 structure} we give the relationship between the classes of $9/8$-closed groups and $5/2$-closed groups.

\section{$5/2$-closed groups}

In this section we summarize $5/2$-closed groups.  A reader who is familiar with \cite{Dobson2023apreprint} will not see anything new in this section.  We will need some basic ideas and notation regarding permutation groups before proceeding.

\begin{definition}\label{imprimitive}
Let $X$ be a set and $G\le S_X$ be transitive.  A subset $B\subseteq X$ is a {\bf block} of $G$ if whenever $g\in G$, then $g(B)\cap B = \emptyset$ or $B$.  If $B = \{x\}$ for some $x\in X$ or $B = X$, then $B$ is a {\bf trivial block}.  Any other block is nontrivial.  Note that if $B$ is a block of $G$, then $g(B)$ is also a block of $B$ for every $g\in G$, and is called a {\bf conjugate block of $B$}.  The set of all blocks conjugate to $B$ is a partition of $X$, called a {\bf block system of $G$}.
\end{definition}

\begin{definition}
Let $X$ be a set and $G\le S_X$ be transitive.  If $N\tl G$, then the set of orbits of $N$ is a block system ${\cal B}$ of $G$.  We say ${\cal B}$ is a {\bf normal block system} of $G$.
\end{definition}

\begin{definition}
Let $G\le S_n$ be transitive with a block system ${\cal B}$.  By $\fix_G({\cal B})$ we denote the subgroup of $G$ that fixes each
block of ${\cal B}$ set-wise.  That is, $\fix_G({\cal B}) = \{g\in G:g(B) = B{\rm\ for\ all\ }B\in{\cal B}\}$.  If ${\cal C}$ is another block system of $G$ and each block of ${\cal B}$ is (properly) contained within a block of ${\cal C}$, we write (${\cal B}\prec{\cal C}$) ${\cal B}\preceq{\cal C}$, and say ${\cal B}$ {\bf refines} ${\cal C}$.  We denote the induced action of $G$ on the block system ${\cal B}$ by $G/{\cal B}$, and the action of an element $g\in G$ on ${\cal B}$ by $g/{\cal B}$.  That is, $g/{\cal B}(B) = B'$ if and only if $g(B) = B'$, and $G/{\cal B} = \{g/{\cal B}:g\in G\}$.  If ${\cal B}\preceq{\cal C}$, then $G/{\cal B}$ has a block system ${\cal C}/{\cal B}$ whose blocks are the set of blocks of ${\cal B}$ whose union are blocks of ${\cal C}$.
\end{definition}

It is easy to see that $\fix_G({\cal B})$ is a normal subgroup of $G$.

\begin{definition}
Let $G\le S_n$ with orbit ${\cal O}$, and $g\in G$. Then $g$ induces a permutation on ${\cal O}$ by restricting the domain of $g$ to ${\cal O}$.  We denote the resulting permutation in $S_{\cal O}$ by $g^{\cal O}$.  The group $G^{\cal O} = \{g^{\cal O}:g\in G\}$ is the  {\bf transitive constituent} of $G$ on ${\cal O}$.
\end{definition}

We now give the necessary definitions to define $5/2$-closed groups.

\begin{definition}
It was shown in \cite[Lemma 1.7]{Dobson2023apreprint} that for $G\le S_n$ transitive with a normal block system ${\cal B}$, for each $B\in{\cal B}$ there is a maximal subgroup of $\fix_G({\cal B})$, denoted $\WStab_G(B)$ (the {\bf wreath} stabilizer of $B$ in $G$), such that $\WStab_G(B)^B = 1$ and for every other block $B'\in{\cal B}$, $\WStab_G(B)^{B'}$ is either transitive or trivial.
\end{definition}

It was also shown in \cite[Lemma 1.9]{Dobson2023apreprint} that $\WStab_G(B)$ behaves exactly like a ``stabilizer" in that the conjugate of a wreath stabilizer of $B$ by $g\in G$ is the wreath stabilizer of $g(B)$ in $G$

\begin{definition}\label{automorphismtooldef}
Let $G$ be a transitive group that has a normal block system ${\cal B}$.  Define a relation $\equiv$ on ${\cal B}$ by $B\equiv B'$ if and only if $\WStab_G(B) = \WStab_G(B')$.
\end{definition}

It was shown in \cite[Lemma 1.11]{Dobson2023apreprint} that $\equiv$ is an equivalence relation on ${\cal B}$, and that the set of unions of equivalence classes of $\equiv$ is a block system of $G$ which is refined by ${\cal B}$.   Note that $B\not\equiv B'$ means that $\WStab_G(B)^{B'}$ is transitive.

\begin{definition}
We call $\equiv$ the {\bf ${\cal B}$-restricting equivalence relation of $G$}, and ${\cal E}$ the {\bf ${\cal B}$-fixer block system of $G$}.
\end{definition}

\begin{notation}
Let $g\in S_n$ and $X\subseteq\Z_n$ (we think of $S_n$ as permuting the elements of $\Z_n$) such that $g(X) = X$.  By $g\vert_X$ we mean the element of $S_n$ such that $g\vert_X(y) = g(y)$ if $y\in X$, while $g\vert_X(y) = y$ if $y\notin X$.
\end{notation}

\begin{definition}
Let $G\le S_n$ be transitive.  For $H\le G$  transitive with normal block system ${\cal B}_H$ let ${\cal E}_{H,{\cal B}}$ be the ${\cal B}_H$-fixer block system of $H$. Suppose that for every transitive subgroup $H\le G$, every normal block system ${\cal B}_H$ of $H$, and every $g\in G$ that fixes each block of ${\cal B}_H$ contained in $E\in{\cal E}_{H,{\cal B}}$ setwise, we have $g\vert_E\in G$.  We will then say that $G$ is {\bf $5/2$-closed}. For a group $G$, the {\bf $5/2$-closure of $G$}, denoted $G^{(5/2)}$, is the intersection of all $5/2$-closed groups which contain $G$.
\end{definition}

The set of $5/2$-closed groups was introduced in \cite{Dobson2023apreprint}, where their elementary properties were studied.  In particular it was shown that the automorphism groups of vertex-transitive digraphs are $5/2$-closed \cite[Theorem 3.6]{Dobson2023apreprint}, and a sufficient condition for the quotient of a $5/2$-closed group to be $5/2$-closed was given in \cite[Theorem 2.3]{Dobson2023apreprint}.  We next continue the development of $5/2$-closed groups by giving an explicit method for constructing the $5/2$-closure of a transitive permutation group, which we will need later.

\section{Constructing $G^{(5/2)}$}

Automorphism groups of digraphs, and consequently $2$-closed groups, can contain ``unexpected" automorphisms, and consequently can be difficult to construct.  We next show that, at least intuitively, $5/2$-closed groups are much easier to construct.  We also show that like $2$-closed groups (see \cite[Theorem 4.11]{Wielandt1969}), the blocks of the $5/2$-closure of a transitive group $G$ are the same as the group $G$.  We begin with some preliminary results.

\begin{lem}\label{blocks with restrictions}
Let $G\le S_n$ be transitive with normal block systems ${\cal B}$ and ${\cal C}$, with ${\cal E}$ the ${\cal C}$-fixer block system of $G$.  Then either ${\cal B}\preceq{\cal E}$ or ${\cal C}\prec{\cal B}$.
\end{lem}

\begin{proof}
The result is trivial if ${\cal E} = \{\Z_n\}$, so we assume ${\cal E}$ has at least two blocks.  Let $x\in\Z_n$, $B_x\in{\cal B}$, $C_x\in{\cal C}$ and $E_x\in{\cal E}$ that contain $x$.  Then $\WStab_G(C_x)^{E_x} = 1$ while $\WStab_G(C_x)^E$, has orbits that are blocks of ${\cal C}$, where $E\in{\cal E}$ with $E\not = E_x$.  If $B_x\subseteq E_x$ then ${\cal B}\preceq{\cal E}$ and we are finished.  Otherwise, there is some $b\in B_x$ that is not contained in $E_x$.  As $\WStab_G(C_x)\le\Stab_G(x)$, the orbit of $\Stab_G(x)$ that contains $b$ contains the block $C_b$ of ${\cal C}$ that contains $b$, and is not contained in $E_x$.  Also, as $B_x$ is a union of orbits of $\Stab_G(x)$ by \cite[Exercise 1.5.9]{DixonM1996}, we have $C_b\subseteq B_x$.  Let $y\in B_x\cap E_x$.  Then there exists $g\in G$ such that $g(b) = y$, and $g(C_b)\subseteq B_x$ contains $y$.  Thus, for every element $z\in B_x\cap E_x$, the block of ${\cal C}$ that contains $z$ is contained in $B_x$.  Hence ${\cal C}\preceq{\cal B}$, and, as $x\not\in C_b$, ${\cal C}\prec{\cal B}$.
\end{proof}

\begin{definition}
For $B\subseteq\Z_n$ and $G\le S_n$, the setwise stabilizer of $B$ in $G$ is $\Stab_G(B) = \{g\in G:g(B) = B\}$, and is a subgroup of $G$.
\end{definition}

\begin{lem}\label{5/2 blocks preserved}
Let $G\le S_n$ be transitive, and $H\le G$ such that $H$ is transitive with a normal block system ${\cal C}$.  Let ${\cal E}$ be the ${\cal C}$-fixing block system of $H$.  A partition ${\cal B}$ of $\Z_n$ is a block system of $G$ if and only if it is a block system of $F = \la G,\gamma\vert_E:\gamma\in\Stab_G(E){\it \ and\ fixes\ each\ block\ of\ }{\cal C}{\it \ in\ }E, E\in{\cal E}\ra$.
\end{lem}

\begin{proof}
As $G\le F$, a block system of $F$ is a block system of $G$.

For the converse, we may assume that ${\cal E}\not = \{\Z_n\}$, as if that is the case, then $G = F$.  Let $\gamma\in\Stab_G(E)$ that fixes each block of ${\cal C}$ contained in $E\in{\cal E}$.  We apply Lemma \ref{blocks with restrictions} to $H$ and consider the two conclusions of this lemma separately.  Suppose ${\cal B}\preceq{\cal E}$.  Then for each block $E\in{\cal E}$, we have either $B\subseteq E$ or $B\cap E = \emptyset$.  In the former case, $(\gamma\vert_E)(B) = \gamma(B)$, while in the latter case, $\gamma\vert_E(B) = B$.  We conclude that $\gamma\vert_E$ maps blocks of ${\cal B}$ to blocks of ${\cal B}$.  If ${\cal C}\prec{\cal B}$, then $\gamma$ fixes every block of ${\cal B}$ contained in $E$, and so $(\gamma\vert_E)(B) = B$ for every $B\in{\cal B}$.  Then ${\cal B}$ is preserved by every element of a generating set of $F$, and the result follows.
\end{proof}

We now wish to fix notation for the next result, which gives a method of constructing the $5/2$-closure of a transitive group.
Let $G\le S_n$ be transitive, and $X = \{{\cal B}_i:1\le i\le t\}$ the set of all normal block systems of some transitive subgroup $H_i\le G$ (the transitive subgroup $H_i$ may vary as $i$ varies), with ${\cal B}_i$-restricting block systems ${\cal E}_i$ of $H_i$.  Let
\begin{eqnarray*}
\clo_{5/2}(G) & = & \la G, \gamma\vert_{E_i}:\gamma\in\Stab_G(E_i){\rm\ that\ fixes\ each\ block\ of\ }{\cal B}_i\\
              &   & {\rm contained\ in\ }E_i, E_i\in{\cal E}_i, 1\le i\le t\ra.
\end{eqnarray*}
For an integer $s$, set $\clo_{5/2}^s(G)$ to be $\clo_{5/2}(\clo_{5/2}(\cdots(\clo_{5/2}(G))\ldots))$ (where $\clo_{5/2}$ is applied $s$ times).  We note that each time $\clo_{5/2}$ is applied, we do not assume that $X$ or $t$ is the same.

The next result is the main result of this section, and shows how to construct $G^{(5/2)}$ from $G$, as well as shows that the block systems of $G$ and $G^{(5/2)}$ are the same.

\begin{thrm}\label{compute G^5/2}
There exists a positive integer $s$ such that $\clo_{5/2}^s(G) = G^{(5/2)}$.  Additionally, the blocks of $G$ and $G^{(5/2)}$ are the same.
\end{thrm}

\begin{proof}
As $S_n$ is finite, there exists a smallest positive integer $s$ such that $\clo^{s+1}_{5/2}(G) = \clo^s_{5/2}(G)$.
Let ${\cal B}$ be a block system of $G$.  Inductively applying Lemma \ref{5/2 blocks preserved}, we see that ${\cal B}$ is a block system of $\clo_{5/2}(G)$.  Inductively applying the previous fact, we see that ${\cal B}$ is a block system of $\clo^s_{5/2}(G)$.  As $\clo_{5/2}^s(G) = \clo_{5/2}^{s+1}(G)$, the block systems ${\cal B}_1,\ldots,{\cal B}_t$ for which there is a transitive subgroup $H_i\le\clo_{5/2}^s(G)$ for which ${\cal B}_i$ is a normal block system of $H_i$ are the block systems ${\cal C}_1,\ldots,{\cal C}_t$ for which there is a transitive subgroup $K_i$ of $\clo_{5/2}^{s+1}(G)$ with ${\cal C}_i$ a normal block system of $K_i$.  It then follows by the definition of $G^{(5/2)}$ that $\clo_{5/2}^s(G) = G^{(5/2)}$.  To finish, we have shown that a block of $G$ is a block of $G^{(5/2)}$.  As $G\le G^{(5/2)}$, a block of $G^{(5/2)}$ is a block of $G$.  Thus the blocks of $G$ and $G^{(5/2)}$ are the same.
\end{proof}

Theorem  \ref{5/2 blocks preserved} gives a theoretical way of computing the $5/2$-closure of any transitive group $G$.  First, find all normal block systems of all transitive subgroups of $G$.  Then calculate $\WStab_G(B)$ for $B\in{\cal B}$ and normal block system ${\cal B}$ found previously. Choose a normal block system ${\cal B}$ of some subgroup $H\le G$ and calculate the ${\cal B}$-fixer block system ${\cal E}$ of $H$, along with the subgroup $F_{{\cal B},E}$ of $G$ which fixes every block of ${\cal B}$ contained in $E\in{\cal E}$.  Then restrict these elements to the appropriate $E\in{\cal E}$.  Repeat this for all block systems previously found.  Then $\clo_{5/2}(G)$ is the group generated by $G$ together with all of the additional elements of $S_n$ that have been found.  Repeat this procedure until no new elements are produced.

\section{The Cayley isomorphism problem}

In this section, we give a very brief overview of the Cayley isomorphism problem, contenting ourselves with only the information that is necessary for this paper.  References are given at the end of this section to publications where much more information can be found.

\begin{definition}
Let $G$ be a group and $S\subseteq G$. Define a {\bf Cayley digraph of $G$},  denoted $\Cay(G,S)$, to be the digraph with vertex set $V(\Cay(G,S)) = G$ and arc set $A(\Cay(G,S)) = \{(g,gs):g\in G, s\in S\}$.  We call $S$ the {\bf connection set of $\Cay(G,S)$}.
\end{definition}

The Cayley isomorphism problem for Cayley graphs began in 1967 when \'Ad\'am conjectured \cite{Adam1967} that two Cayley graphs $\Cay(\Z_n,S)$ and $\Cay(\Z_n,T)$ are isomorphic if an only if there is an element $m\in\Z_n^*$ (a {\bf multiplier}) such that $mS = T$.  Note that $mS = T$ is equivalent to $\Cay(\Z_n,S)$ and $\Cay(\Z_n,T)$ are isomorphic by a group automorphism of $\Z_n$.  Also, the image of a Cayley digraph of $G$ under a group automorphism of $G$ is a Cayley digraph of $G$ \cite[Lemma 1.2.15]{Book}.  So \'Ad\'am was essentially conjecturing that the smallest list of isomorphisms (the group automorphisms of $G$) to check for isomorphism would determine isomorphism between Cayley graphs of $\Z_n$.  Elspas and Turner \cite{ElspasT1970} in 1970 quickly showed that this conjecture is false for both graphs and digraphs, and the conjecture changed into the problem of which groups $G$ have the property that any two Cayley (di)graphs of $G$ are isomorphic if and only if they are isomorphic by a group automorphism of $G$?

\begin{definition}
Let $G$ be a group, and $S\subseteq G$.  We say $\Cay(G,S)$ is a {\bf CI-digraph} of $G$ if and only if whenever $T\subseteq G$ then $\Cay(G,S)\cong\Cay(G,T)$ if and only if there is $\alpha\in \Aut(G)$ such that $\alpha(\Cay(G,S)) = \Cay(G,T)$.
\end{definition}

\begin{definition}\label{leftregularrepresentation}
Let $G$ be a group and $g\in G$.  Define $g_L\colon G\to G$ by $g_L(x) = gx$.  The map $g_L$ is a {\bf left translation of $G$}. The {\bf left regular representation of $G$}, denoted $G_L$, is $G_L = \{g_L:g\in G\}$.
\end{definition}

It is straightforward to show $G_L$ is a group, and $G_L\le\Aut(\Cay(G,S))$ for every $S\subseteq G$.

Of course, one can also ask the same question for other types of combinatorial objects once we have a natural generalization of the notion of a Cayley graph for them.  Sabidussi \cite[Lemma 4]{Sabidussi1958} has shown that a graph $\Gamma$ is isomorphic to a Cayley graph of $G$ if and only if $\Aut(\Gamma)$ contains a regular subgroup isomorphic to $G$.  Hence the next definition extends the notion of a Cayley graph to other combinatorial objects.

\begin{definition}
A {\bf Cayley object} $X$ of a group $G$ in a class ${\cal K}$ of combinatorial objects is one in which $G_L\le\Aut(X)$, the automorphism group of $X$.
\end{definition}

We next generalize the notion of a CI-digraph of $G$ to arbitrary combinatorial objects.

\begin{definition}
For a Cayley object $X$ of $G$ in some class of combinatorial objects ${\cal K}$, we say that $X$ is a {\bf CI-object of $G$} if and only if whenever $X'$ is another Cayley object of $G$ in ${\cal K}$, then $X$ and $X'$ are isomorphic if and only if $\alpha(X) = X'$ for some $\alpha\in\Aut(G)$.
\end{definition}

Babai characterized CI-objects of a group $G$ \cite[Lemma 3.1]{Babai1977}.

\begin{lem}\label{CI}
Let $X$ be a Cayley object of $G$ in some class ${\cal K}$ of combinatorial objects.  Then the following are equivalent:
\begin{enumerate}
\item\label{Ciso1} $X$ is a CI-object of $G$ in ${\cal K}$,
\item\label{Ciso2} whenever $\phi\in S_G$ such that $\phi^{-1}G_L\phi\le\Aut(X)$, $G_L$ and $\phi^{-1}G_L\phi$ are conjugate in $\Aut(X)$.
\end{enumerate}
\end{lem}

\begin{definition}
Let $G$ be a group and ${\cal K}$ a class of combinatorial objects.  We say that $G$ is a {\bf CI-group with respect to ${\cal K}$} if every Cayley object of $G$ in ${\cal K}$ is a CI-object of $G$.
\end{definition}

In this paper, we will mainly be concerned with Cayley objects of a cyclic group.  Such Cayley objects are called {\bf circulants}.

Much work on the problem of determining which groups are CI-groups with respect to graphs has been done.  We refer the reader to \cite{Li2002} for a survey paper on the problem, as well as to \cite[Theorem 5.2]{DobsonMS2022} for the current list of possible CI-groups with respect to digraphs.  See also \cite[Chapter 7]{Book}.

\section{Some permutation group theory}

In this section, we will summarize much of the permutation group theory in the literature which is directly applicable to the problems considered here.  Also, we prove one additional result which will be useful.  We start with some definitions.

\begin{definition}\label{mstepdefin}
Let $n = p_1^{a_1}p_2^{a_2}\cdots p_r^{a_r}$ be the prime power decomposition of $n$, and define $\Omega\colon {\mathbb N}\to {\mathbb N}$ by $\Omega(n) = \sum_{i=1}^ra_i$ (so $m = \Omega(n)$ is the number of prime divisors of $n$ with repetition allowed).  A transitive group $G\le S_n$ is {\bf $m$-step imprimitive} if there exists a sequence of block systems ${\cal B}_0 \prec{\cal B}_1\prec \ldots \prec{\cal B}_m$ of $G$, where ${\cal B}_0$ is the set of all singleton sets and ${\cal B}_m$ is $\{\Z_n\}$, and if $B_{i+1}\in{\cal B}_{i+1}$ and $B_i\in{\cal B}_i$, then $\vert B_{i+1}\vert/\vert B_i\vert$ is prime, $0\le i\le m - 1$.  (Technically this last condition is not strictly necessary as ${\cal B}_{i+1}$ is not, by definition, equal to ${\cal B}_i$, but we list it anyway to emphasize the property).  If, in addition, each ${\cal B}_i$ is normal, then we say that $G$ is {\bf normally $m$-step imprimitive}.  We call ${\cal B}_0\prec{\cal B}_1\prec\ldots\prec{\cal B}_m$ an {\bf $m$-step imprimitivity sequence of $G$}, and a {\bf normal $m$-step imprimitivity sequence of $G$} if each ${\cal B}_i$ is a normal block system of $G$.
\end{definition}

\begin{definition}
Let $G\le S_X$ and $H\le S_Y$.  Define the {\bf wreath product of $G$ and $H$}, denoted $G\wr H$, to be the set of all permutations of $X\times Y$ of the form $(x,y)\mapsto (g(x),h_x(y))$, where $g\in G$ and each $h_x\in H$.
\end{definition}

It is straightforward to show that $G\wr H$ is a group.  See \cite[\S 4.2]{Book} for examples concerning the wreath product of groups.

Let $\la x\ra = (\Z_n)_L$ and $\la y\ra\le S_n$ a conjugate of $\la x\ra$.  There are many results in the literature dealing with how to conjugate $\la y\ra$ by an element of $\la x,y\ra$ to make it ``closer" to $\la x\ra$.   In view of Lemma \ref{CI}, these are exactly the types of results we would want when considering the CI-problem for cyclic groups and any class of combinatorial object.  Specifically, applying (in this order) \cite[Theorem 4.9]{Muzychuk1999}, \cite[Theorem 1.8]{Muzychuk1999}, \cite[Lemma 19]{Dobson2008a}, and \cite[Lemma 24]{Dobson2008a}, one obtains the following result:

\begin{lem}\label{already known}
Let $n$ be a positive integer with prime power decomposition $n = p_1^{a_1}p_2^{a_2}\cdots p_r^{a_r}$ where each $a_i\ge 1$, $1\le i\le r$, and $p_1 > p_2 > \ldots > p_r$.  Let $\la x\ra = (\Z_n)_L$, and $y\in S_n$ generate a regular cyclic subgroup.  Let $m = \Omega(n)$.  Then there exists $\delta\in\la x,y\ra$ such that $\la x,\delta^{-1}y\delta\ra$ satisfies the following conditions:
\begin{enumerate}
\item $\la x,\delta^{-1}y\delta\ra$ is normally $m$-step imprimitive with normal imprimitivity sequence ${\cal B}_0\prec{\cal B}_1\prec\cdots\prec{\cal B}_m$ and $\vert B_{i+1}\vert/\vert B_i\vert \ge \vert B_{i + 2}\vert/\vert B_{i + 1}\vert$, $0\le i\le m-2$,
\item Set $b_0 = 1$, and for $1\le i\le r$, set $b_i = \sum_{j=1}^ia_i$.  If $b_i - 1\le k\le b_{i + 1} - 1$, $B_k\in{\cal B}_k$, $B_{k + 1}\in{\cal B}_{k+1}$, then $\vert B_{k+1}/\vert B_k\vert = p_{i + 1}$,
\item $\la x,\delta^{-1}y\delta\ra$ is solvable,
\item $\la x,\delta^{-1}y\delta\ra$ is permutation isomorphic to a subgroup of $$\AGL(1,p_r)^{[a_r]}\wr\AGL(1,p_{r-1})^{[a_{r-1}]}\wr\ldots\wr\AGL(1,p_1)^{[a_1]},$$ where $\AGL(1,p_i)^{[a_i]}$ is the wreath product of $\AGL(1,p_i)$ with itself $a_i$ times, and
\item\label{part 5} if $p_1$ is odd, then $\fix_{\la x,\delta^{-1}y\delta\ra}({\cal B}_1)$ has a Sylow $p_1$-subgroup or order at least $p_1^2$, or $\fix_{\la x,\delta^{-1}y\delta\ra}({\cal B}_{a_1})$ has a unique cyclic Sylow $p_1$-subgroup $P = \la x^{n/p_1^{a_1}}\ra$, $P\le\C(\la x,\delta^{-1}y\delta\ra)$, and $\fix_{\la x,\delta^{-1}y\delta\ra}({\cal B}_{a_1}) = P$.
\end{enumerate}
\end{lem}

Note that in (\ref{part 5}) if $p_1$ is not odd, then $r = 1$ and $n$ is a power of $2$.  Our next result removes the hypothesis that $p_1$ is odd in the last part of the previous result.

\begin{lem}\label{pimpernel found}
Let $k\ge 2$ and $x,y\in S_{2^k}$ such that $\la x\ra$ and  $\la y\ra$ are regular cyclic subgroups and $\la x,y\ra$ is a $2$-group.  Let ${\cal B}_0\prec{\cal B}_1\prec\cdots\prec{\cal B}_k$ be a normal $k$-step imprimitivity sequence of $\la x,y\ra$.  Then $\la x,y\ra = \la x\ra$ or $\vert \fix_{\la x,y\ra}({\cal B}_1)\vert \ge 4$.
\end{lem}

\begin{proof}
Notice that $x^{2^{k-1}}, y^{2^{k-1}}\in\fix_{\la x,y\ra}({\cal B}_1)$.  Additionally, as both $\la x^{2^{k-1}}\ra$ and $\la y^{2^{k-1}}\ra$ are semiregular subgroups of order $2$ that fix each block of ${\cal B}_1$ and permutate the elements of a block of ${\cal B}_1$ as a transposition, $\la x^{2^{k-1}}\ra = \la y^{2^{k-1}}\ra$.  Let $1\le j\le k$ be maximum such that $\la x^{2^{k-j}}\ra = \la y^{2^{k-j}}\ra$.  By the previous argument, $j$ exists.  The result follows if $j = k$, so we assume without loss of generality that $1 \le j < k$.  Then $\la x^{k-j}\ra = \la y^{k-j}\ra$ commutes with every element of $\la x\ra$ and $\la y\ra$, and so $x^{k-j}\in \C(\la x,y\ra)$.  Additionally, as $\la x^{k-j}\ra\in \C(\la x,y\ra)$, we see that $\Stab_{\la x,y\ra}(B_j)^{B_j}$ commutes with $\la x^{k-j}\ra^{B_j}$ for every $B_j\in{\cal B}_j$.  As a transitive abelian group is self-centralizing \cite[Corollary 2.2.18]{Book}, we see that $\Stab_{\la x,y\ra}(B_j)^{B_j} = \la x^{k-j}\ra^{B_j}$ for every $B_j\in{\cal B}_j$.   Then $\la x,y\ra$ is permutation isomorphic to a subgroup of $(\la x,y\ra/{\cal B}_j)\wr(\Z_{2^{k-j}})_L$ by the Embedding Theorem \cite[Theorem 4.3.1]{Book}.

Now, by arguments at the beginning of the proof of this result, $\la x^{k-j-1}\ra/{\cal B}_j = \la y^{k-j-1}\ra/{\cal B}_j$.
Then $1\not = x^{k-j-1}y^{k-j-1}\in\fix_{\la x,y\ra}({\cal B}_j)$ as, by hypothesis, $\la x^{k - j - 1}\ra\not = \la y^{k - j - 1}\ra$.  As $\Stab_{\la x,y\ra}(B_j)^{B_j} = \la x^{k-j}\ra^{B_j}$ for every $B_j\in{\cal B}_j$, we see that $$\la x^{k-j-1}y^{k-j-1}\ra^{B_j}\in \la x^{k - j}\ra^{B_j}$$ for every $B_j\in{\cal B}_j$.  Then there exists $a\in\Z$ such that $[(x^{k-j})^a(x^{k-j-1}y^{k-j-1})]^{B_j} = 1$ for some fixed $B_j\in{\cal B}_j$.  But as $\la x^{k - j -1}\ra\not = \la y^{k - j - 1}\ra$, it cannot be the case that $(x^{k-j})^a(x^{k-j-1}y^{k-j-1}) = 1$.  Then there is $B_j'\in{\cal B}_j$ such that $$[(x^{k-j})^a(x^{k-j-1}y^{k-j-1})]^{B_j'}\not = 1.$$  As $\Stab_{\la x,y\ra}(B_j)^{B_j} = \la x^{k-j}\ra^{B_j}$ for every $B_j\in{\cal B}_j$,  $\Stab_{\la x,y\ra}(B_j')^{B_j'}$ is a regular cyclic subgroup of order dividing $2^j$, and so contains a unique subgroup of order $2$.  We conclude that if $(x^{k-j})^a(x^{k-j-1}y^{k-j-1})$ has order $2^\ell$, then $$[(x^{k-j})^a(x^{k-j-1}y^{k-j-1})]^{2^{\ell - 1}}\in\fix_{\la x,y\ra}({\cal B}_1),$$ has order $2$, and has a fixed point.  Thus $\fix_{\la x,y\ra}({\cal B}_1)\not = \la x^{2^{k-1}}\ra$, and the result follows.
\end{proof}

\section{Configurations and $9/8$-closed groups}

In this section we consider a smaller class of groups than $5/2$-closed groups, which we will call $9/8$-closed groups (and will define another used later).  The reasoning for the various choices of names of families of groups is as follows.  In Wielandt's $k$-closure hierarchy, the class of groups always becomes larger as $k$ grows by \cite[Theorem 5.10]{Wielandt1969}.  The $1$-closed groups are just the symmetric groups \cite[Theorem 5.11]{Wielandt1969}.  So, following Wielandt, we want the family of groups to ``grow" as the number grows.  We put grow in quotes as we mean intuitively, not necessarily absolutely.  Thus we mean that the set of all $2$-closed groups is by and large ``bigger" than the set of $9/8$-closed groups.  But not every $9/8$-closed group is $2$-closed.  For classes outside of Wielandt's hierarchy though, inclusion does hold.  We use the equivalence relation $\equiv$ as defined in Definition \ref{automorphismtooldef} but impose some additional conditions (thus maintaining inclusion).

\begin{definition}
Let $G\le S_n$ be transitive and $5/2$-closed.  We will say that $G$ is $9/8$-closed if whenever $H\le G$ is transitive with normal block system ${\cal B}$, then the ${\cal B}$-fixing block system ${\cal E}_{H,{\cal B}}$ of $H$ is $\{\Z_n\}$.  Let $K_H\le G$ be the largest subgroup that has ${\cal B}$ as a block system.  We say $G$ is $5/4$-closed if whenever $H\le G$ is transitive with normal block system ${\cal B}$ then ${\cal E}_{K_H,{\cal B}}$ is either ${\cal B}$ or $\Z_n$.   We say $G$ is $3/2$-closed if $G$ is $5/4$-closed and whenever ${\cal E}_{H,{\cal B}} = {\cal B}$, then $\fix_K({\cal B})^B = S_B$ for every $B\in{\cal B}$.
\end{definition}

Notice that in the above definition, for the $5/4$-closure and $3/2$-closure we do not compute the ${\cal B}$-fixing block system of $H$, but of the largest subgroup $K_H$ of $G$ which has ${\cal B}$ as a block system.

Clearly a $9/8$-closed group is $5/4$-closed, and a $5/4$-closed group is $3/2$-closed.  Also observe that a transitive subgroup of a $9/8$-closed group is also $9/8$-closed.  At present, we have no applications for $5/4$-closed groups, but they are an obvious class between $9/8$-closed and $3/2$-closed groups.  We will consider $3/2$-closed groups in the next section.

Our next goal is to define most of the different combinatorial objects which we will be considering in this section.  The following definitions and results are mostly taken from \cite{Dobson2023apreprint}.

\begin{definition}\label{tuple def}
Let $X$ be a finite set and $T$ a subset of $\cup_{i=1}^n X^i$ for some $n < \vert X\vert$.  Hence $T$ is a set of $k_i$-tuples with elements in $X$ for some set of integers $k_1,\ldots,k_r < n$.  We will call $T$ a {\bf tuple system} on $X$.  Let $S = \{\{u_1,\ldots,u_t\}:(u_1,\ldots,u_t)\in T\}$, so $S$ is the set of coordinates of each element of $T$, and we call it the  {\bf set system corresponding to $T$}.  Let $m\ge 0$ be an integer.  We say a set system $S$ is $m$-intersecting if for every $U_1,U_2\in S$ with $U_1\not = U_2$, $\vert U_1\cap U_2\vert\le m$.  A tuple system $T$ will also be called $1$-intersecting if its corresponding set system is $1$-intersecting.
\end{definition}

\begin{example}
The arc set of a finite digraph $\Gamma$ is a tuple system, with corresponding set system the set of edges of the underlying simple graph.  Also, observe that $A(\Gamma)$ is $1$-intersecting, as two different arcs can only share at most one head or tail.
\end{example}

\begin{definition}
Let $P$ be a set whose elements are called points, and $L$ be a set whose elements are called lines.  An {\bf incidence relation} is an ordered triple $(P,L,{\cal I})$ where ${\cal I}\subseteq P\times L$.  We say the point $p$ is on the line $\ell$ or the line $\ell$ contains the point $p$ if $(p,\ell)\in {\cal I}$.
\end{definition}

\begin{definition}
A {\bf configuration} is an incidence relation $(P,L,{\cal I})$ for which there are positive integers $q$ and $k$ such that the following conditions hold:
\begin{itemize}
\item each of the points is on exactly $q$ lines,
\item each line contains exactly $k$ points, and
\item two points are on at most one line and two lines contain at most one point in common.
\end{itemize}
\end{definition}

Thus a configuration is a $1$-intersecting set system.  For more information on configurations, see \cite{Grunbaum2009}, \cite{PisanskiS2013}, or \cite[Section 6.4]{Book}.

\begin{definition}
Let $T$ be a tuple (set) system.  We say $T$ is {\bf colored} if each tuple (set) in $T$ has been identified with a color (or integer).  Hence there is a function $f:T\to {\mathbb N}$ and $f(t)$ is simply the color of the tuple (set) $t\in T$.
\end{definition}

We may think of an uncolored tuple (set) system $T$ as a colored tuple (set) system simply by assigning the same color to every element of $T$.

\begin{definition}
Let $T$ and $T'$ be colored tuple (set) systems defined on $X$ and $Y$, respectively.  An {\bf isomorphism} from $T$ to $T'$ is a bijection $f:X\to Y$ that preserves tuples (sets) and colors.  That is, $f(t)\in T'$ if and only if $t\in T$ and for every color $i$, the image of the set of all tuples (sets) in $T$ colored $i$ is the set of all tuples (sets) colored $i'$ in $T'$, for some color $i'$.  An {\bf automorphism} of $T$ is an isomorphism from $T$ to $T$ that fixes each color, and the set of all automorphisms of $T$ is the {\bf automorphism group} of $T$, denoted $\Aut(T)$.  A colored tuple (set) system $T$ is {\bf point-transitive} if $\Aut(T)$ is transitive on $X$.
\end{definition}

In \cite[Example 3.7]{Dobson2023apreprint} it was shown that a configuration can be identified with a tuple system in such a way that the automorphism group of the configuration is the same as the automorphism group of the tuple system.  Thus results on automorphism groups of tuple systems are also results about automorphism groups of configurations and digraphs.

\begin{lem}\label{Aut T and C}
Let $T$ be a colored tuple system with $C$ the set system corresponding to $T$.  Then $\Aut(T)\le \Aut(C)$.  Consequently, if $\Aut(C)$ is $9/8$-closed, then $\Aut(T)$ is $9/8$-closed.
\end{lem}

\begin{proof}
Let $\gamma\in\Aut(T)$, with $t = (t_1,\ldots,t_k)\in T$ of color $k$, so that $\{t_i:1\le i\le k\}\in C$.  Then $\gamma(t_1,\ldots,t_k) = (\gamma(t_1),\ldots,\gamma(t_k))$ has color $k$, and so $\gamma(\{t_i:1\le i\le k\}) = \{\gamma(t_i):1\le i\le k\}\in C$.  As $\gamma$ and $t$ are arbitrary, $\gamma\in\Aut(C)$ has color $k$.  Thus $\Aut(T)\le\Aut(C)$.  It is now clear that if for every transitive subgroup $H\le\Aut(C)$ with normal subgroup block system ${\cal B}$ has ${\cal B}$-fixing block system with one block, then the same is true for $\Aut(T)$.  Thus if $\Aut(C)$ is $9/8$-closed, then so $\Aut(T)$.
\end{proof}

The next result is \cite[Theorem 3.6]{Dobson2023apreprint}.  It gives a partial relationship between $5/2$-closed groups and automorphism groups of combinatorial objects.

\begin{thrm}\label{5/2 closed}
Let $T$ be a point-transitive colored $1$-intersecting tuple system.  Then $\Aut(T)$ is $5/2$-closed.
\end{thrm}

We next seek partial relationships between $9/8$-closed groups and automorphism groups of combinatorial objects.  The next definition gives the largest class of combinatorial objects for which we can show that all objects in the class have automorphism groups that are $9/8$-closed.

\begin{definition}
A {\bf partial Sylvester-Gallai design} is an incidence relation $(P,L,{\cal I})$ such that
\begin{enumerate}
\item any two points are contained on at most one line, and
\item each line has at least three points.
\end{enumerate}
\end{definition}

Sylvester-Gallai designs were introduced by Kelly and Nwankpa \cite{KellyN1973} as generalizations of Sylvester-Gallai configurations, and have the additional property that any two points determine a line.

\begin{definition}
A colored tuple system is {\bf connected} if its corresponding set system $C$ is connected.  That is, for any two vertices $u,v$ of the corresponding set system, there is a sequence of sets $e_1,\ldots,e_r$ in $C$ such that $u\in e_1$, $v\in e_r$, and $e_i\cap e_{i + 1}\not = \emptyset$, $1\le i\le r-1$.
\end{definition}

\begin{thrm}\label{9/8 object}
Let $T$ be a colored tuple system whose set system $C$ corresponding to $T$ is a connected partial Sylvester-Gallai design with point-transitive automorphism group. Then $\Aut(T)\le\Aut(C)$ are both $9/8$-closed.
\end{thrm}

\begin{proof}
By Lemma \ref{Aut T and C} it suffices to show that $\Aut(C)$ is $9/8$-closed. Let $C$ satisfy the hypothesis, with $G\le\Aut(C)$ transitive with a normal block system ${\cal B}$.  Towards a contradiction, suppose $\WStab_G(B)\not = 1$ for some $B\in{\cal B}$ with $B'\in{\cal B}$ such that $\WStab_G(B)^{B'}$ is transitive.   Then the ${\cal B}$-fixing block system of $G$ has at least two blocks.  As $C$ is connected, there is some line with points in two different blocks of ${\cal E}$.   Let $E$ be the equivalence class of $\equiv$ that contains $B$ and $E'$ the equivalence class of $\equiv$ that contains $B'$.  We may assume without loss of generality (by relabeling if necessary) that some line of $C$ contains a point $x$ in $B$ and a point $y$ in $B'$.  Then there exists $w\in \WStab_G(B)$ such that $w(y)\not = y$ (and of course $w(x) = x$).

As $C$ is a $1$-intersecting set system $\Aut(C)$ is $5/2$-closed by Theorem \ref{5/2 closed}.  Then $w\vert_{E'}\in \Aut(C)$, and so we may assume without loss of generality that $w = w\vert_{E'}$.  As any two points are on at most one line, $L$ cannot contain any point other than $x$ that is not in $E'$.  As each line contains at least three points, there is some $x\not = z\not = y$ on $L$, and so $z\in E'$.  So $w(L)$ is a line that contains at least two points in $E'$, namely $y$ and $z$, and exactly one point not in $E'$, namely $x$.  As $\fix_G({\cal B})^B$ is transitive and $\Aut(C)$ is $5/2$-closed, there is $v\in\Aut(C)$ such that $v(x)\not = x$ but $v^{E'} = 1$.  Then $y$ and $z$ are contained on two different lines, a contradiction.
\end{proof}

We next give a sufficient condition to ensure that the automorphism group of a weakly connected vertex-transitive digraph is $9/8$-closed (we note that the automorphism group of disconnected vertex-transitive digraph is never $9/8$-closed as its automorphism group is a wreath product).  We will need a couple of definitions.

\begin{definition}
Let $G$ be a group, $H \leq G$, and $S \subset G$ such that $S\cap H = \emptyset$ and $HSH = S$. Define a digraph $\text{Cos}(G,H,S)$ with vertex set $V(\text{Cos}(G,H,S)) = G/H$ the set of left cosets of $H$ in $G$, and arc set $A(\text{Cos}(G,H,S)) = \{(gH,gsH):g\in G{\rm\ and\ }s\in S\}$. The digraph $\text{Cos}(G,H,S)$ is called the \textbf{double coset digraph of $G$} with \textbf{connection set $S$} (or $HSH$).
\end{definition}

Sabidussi has shown \cite[Theorem 2]{Sabidussi1964} that every vertex-transitive digraph is isomorphic to a double coset digraph (see \cite[Theorem 1.3.9]{Book} for a more modern proof of this result).

\begin{definition}
Let $n$ be a positive integer.  By $\vec{K}_{n,n}$ we mean the complete bipartite graph $K_{n,n}$ oriented so that every arc points from one bipartition set to the other.
\end{definition}

\begin{thrm}\label{graph 9/8 closed}
Let $G$ be a group, $H\le G$, and $S\subset G$ such that $S\cap H = \emptyset$ and $S = HSH$.  Let $n = [G:H]$ and $p$ the smallest prime divisor of $n$.  If $\Cos(G,H,S)$ is weakly connected and has no subdigraph isomorphic to a $\vec{K}_{p,p}$, then $\Aut(\Cos(G,H,S))$ is $9/8$-closed.  Consequently, if $\Cos(G,H,S)$ is a graph of girth at least $5$, then $\Aut(\Cos(G,H,S))$ is $9/8$-closed.
\end{thrm}

\begin{proof}
Let $K\le\Aut(\Cos(G,H,S))$ be transitive with a normal block system ${\cal B}$.  Suppose $\WStab_K(B)\not = 1$ for some $B\in{\cal B}$.  Then there is $B\not = B'\in{\cal B}$ such that $\WStab_K(B)^{B'}$ is transitive.  As $\Cos(G,H,S)$ is weakly connected, we may choose $B$ and $B'$ such that some arc $a = (gH,gsH)$ in $\Cos(G,H,S)$ has $gH\in B$ and $gsH\in B'$.  As $\WStab_K(B)^{B'}$ is transitive, the arc $(gH,kH)\in A(\Cos(G,H,S))$ for every left coset $kH$ of $H$ contained in $B'$.  As ${\cal B}$ is normal, we see that $(\ell H,kH)\in A(\Cos(G,H,S))$ for every left coset $\ell H$ in $B$ and $kH$ in $B'$.  Hence $\Cos(G,H,S)$ contains a $\vec{K}_{b,b}$, where $b = \vert B\vert$.  As the order of a block of $K$ divides $n$, we see that $b\ge p$.  Hence $\Cos(G,H,S)$ contains a subdigraph isomorphic to $\vec{K}_{p,p}$, a contradiction.  Thus $\WStab_K(B) = 1$.  As $K$ and ${\cal B}$ are arbitrary, $\Aut(\Cos(G,H,S))$ is $9/8$-closed.

In particular, if $\Cos(G,H,S)$ is a graph of girth at least $5$, then it has no subdigraph isomorphic to the cycle of length $4$, which is isomorphic to $K_{2,2}$.  As every prime divisor of $n$ is at least $2$, and any subdigraph of $\Cos(G,H,S))$ isomorphic to $\vec{K}_{2,2}$ induces a $K_{2,2}$ in $\Cos(G,H,S)$, we see by the previous argument that $\Aut(\Cos(G,H,S))$ is $9/8$-closed.
\end{proof}

We now give a preliminary result and some definitions in preparation for the main result of this section.  This main result will be used to solve the isomorphism problem for many circulant combinatorial objects we have seen.

\begin{lem}\label{product 9/8-closed}
Let $H\le S_m$ and $K\le S_k$ be transitive.  Let $G = H\times K$ with the canonical action of $H\times K$ on $\Z_m\times\Z_k$.  That is, $(h,k)(i,j) = (h(i),k(j))$ for $h\in H$, $k\in K$, $i\in\Z_m$, and $j\in\Z_k$.  If $G$ is $9/8$-closed, then $H$ and $K$ are $9/8$-closed.
\end{lem}

\begin{proof}
We will show that $H$ is $9/8$-closed, the argument for $K$ being analogous.  Let ${\cal C}$ be a normal block system of $H$.  Let $L = \fix_H({\cal C})$.  Viewing $L$ as an internal subgroup of $H\times K$ acting canonically on $\Z_m\times \Z_k$, we see $L\tl G$ whose set of orbits ${\cal D}$ is a block system of $H\times K$.  As $H\times K$ acts canonically on $\Z_m\times\Z_k$, we see that a block of ${\cal D}$ has the form $C\times\{j\}$, where $C\in{\cal C}$ and $j\in\Z_k$.  Then $\fix_H({\cal C})\times 1_K\le \fix_{H\times K}({\cal D})$.
As $G$ is $9/8$-closed, the ${\cal D}$-fixing block system ${\cal E}_{\cal D}$ of $G$ is $\{\Z_m\times\Z_k\}$.  If the ${\cal C}$-fixing block system ${\cal E}_{\cal C}$ of $H$ is not $\Z_m$, with say $1\not = \ell\in\fix_H({\cal C}) = L$ such that $\ell\in\WStab_H(C)$ for some $C\in{\cal C}$, then $\ell\times 1_K$ is contained in $\WStab_{H\times K}(C\times\{j\})$ for every $j\in\Z_k$ and is not the identity.  Thus $\WStab_{H\times K}(C\times\{j\})\not = 1$, and $H$ is not $9/8$-closed, a contradiction.
\end{proof}

\begin{definition}
Let $G$ be a group.  A subgroup $H\le G$ is {\bf pronormal} in $G$ if for every $g\in G$ there is a $k\in\la H,g^{-1}Hg\ra$ such that $k^{-1}g^{-1}Hgk = H$.
\end{definition}

\begin{definition}
Let $\pi$ be a set of prime numbers.  A group $G$ is called a {\bf $\pi$-group} if every element in $G$ has order a product of powers of primes in $\pi$.
\end{definition}

So a group $G$ is a $\pi$-group if a prime $p$ divides the order of an element of $G$ only if $p\in\pi$.

\begin{definition}
Let $\pi$ be a set of primes, and $G$ a group of order $n$.  We say that $H\le G$ is a {\bf Hall $\pi$-subgroup} of $G$ if $H$ is a $\pi$-group and $\vert G\vert/\vert H\vert$ is relatively prime to $p$ for every $p\in \pi$.  That is, if $p^a\vert n$, $a\ge 0$, then $p^a$ divides $\vert H\vert$ for every $p\in\pi$.
\end{definition}

The next result is the main theoretical result of this section.  We note that if $G\le S_n$ contains a regular cyclic subgroup, then every block system of $G$ is normal \cite[Theorem 2.2.9]{Book}.  So in the following result, ``block system" is synonymous with ``normal block system".

\begin{thrm}\label{configuration result}
Let $R = \la x\ra\le S_n$ be a regular cyclic subgroup and $y$ a conjugate of $x$ in $S_n$.  Suppose that for every conjugate $y'$ of $y$ in $\la x,y\ra$ and every block system ${\cal B}$ of $\la x,y'\ra$, we have that the ${\cal B}$-fixing block system is $\{\Z_n\}$.  Then $\la x\ra$ and $\la y\ra$ are conjugate in $\la x,y\ra$.  Consequently, if $G\le S_n$ is transitive, $9/8$-closed, and contains a regular cyclic subgroup $R$, then $R$ is pronormal in $G$.
\end{thrm}

\begin{proof}
Let $R = \la x\ra$ and $\la y\ra$ be regular cyclic subgroups that satisfy the hypothesis of the result.  We may assume without loss of generality that $y$ has been conjugated so as to satisfy the conclusion of Lemma \ref{already known}.  We will also use all of the notation from both the statement and conclusion of Lemma \ref{already known}.  Set $H = \la x,y\ra$.

We proceed by induction on $r$, the number of distinct prime factors of $n$.  If $r = 1$, then $n$ is a prime power. Suppose that $\fix_H({\cal B}_1)$ has order at least $p^2$.  By \cite[Lemma 1.24]{Dobson2023apreprint} $\WStab_H({\cal B}_1) = \PStab_{\fix_H({\cal B})}(B)$ has order at least $p$, where $B\in{\cal B}$ and $\PStab_{\fix_H({\cal B})}(B)$ is point-wise stabilizer of $B$ in $\fix_H({\cal B})$.  But then the ${\cal B}_1$-fixing equivalence class ${\cal E}_1$ of $G\ge H$ is not $\{\Z_n\}$, a contradiction.  Thus by either Lemma \ref{already known} if $n$ is odd and Lemma \ref{pimpernel found} if $n$ is even, we see that $G$ is not $9/8$-closed, a contradiction which establishes the result.

Now assume that the result holds for all $n$ with $r\ge 1$.  Let $n$ be a positive integer such that $n$ has $r+1$ distinct prime divisors, and $y\in S_n$ generate a regular cyclic subgroup such that $\la x,y\ra$ satisfies the conclusion of Lemma \ref{already known}.  As $\la x,y\ra$ is solvable, $\la x,y\ra/{\cal B}_{a_1}$ is also solvable.  Let $\pi = \{p_i:2\le i\le r + 1\}$.  Then $\la x\ra/{\cal B}_{a_1}$ and $\la y\ra/{\cal B}_1$ are solvable of degree relatively prime to $p_1$, and so $\la x\ra/{\cal B}_{a_1}$ and $\la y\ra/{\cal B}_{a_1}$ are contained in $\pi$-subgroups $\Pi_1$ and $\Pi_2$ of $\la x,y\ra/{\cal B}_{a_1}$, respectively, by Hall's Theorem \cite[Proposition 7.14]{Hungerford1980}. Again by Hall's Theorem, as $\la x,y\ra/{\cal B}_{a_i}$ is solvable, after conjugation of $y$ by an element of $\la x,y\ra$, we may assume without loss of generality that $\Pi_2 = \Pi_1$ and so $\la x,y\ra/{\cal B}_{a_1}$ is a $\pi$-subgroup.  We conclude that a Sylow $p_1$-subgroup of $\la x,y\ra$ is contained in $\fix_H({\cal B}_{a_1})$.

As the ${\cal B}_1$-fixing block system of $\la x,y\ra$ is, by hypothesis, $\{\Z_n\}$, we see that a Sylow $p_1$-subgroup of $\fix_H({\cal B}_1)$ is cyclic.  By Lemma \ref{already known} (\ref{part 5}) if $p_1$ is odd or Lemma \ref{pimpernel found} if $p_1 = 2$, we have that $P = \la x^{n/p_1^{a_1}}\ra = \la y^{n/p_1^{a_1}}\ra$ is a Sylow $p_1$-subgroup of $\la x,y\ra$ and $P$ is central in $H$.  In particular, $P$ is contained in the center of its normalizer in $H$.

We now apply Burnside's Transfer Theorem \cite[Theorem 8.2.10]{Book} to obtain that $P$ contains a normal $p_1'$-complement in $\la x,y\ra$, which we will call $K$.  Then $K,P\tl \la x,y\ra$, $\la x,y\ra = \la K,P\ra$, and as $P$ and $K$ have relatively prime orders, $K\cap P = 1$.  Thus $\la x,y\ra\cong K\times P$.  By Lemma \ref{product 9/8-closed}, we see that $K$ is $9/8$-closed.  Hence as $K$ has degree $n/p_1^{a_1}$ with $r$ distinct factors, we may assume by the induction hypothesis, after conjugation of $\la y\ra$ by an appropriate element of $\la x,y\ra$, that $\la x,y\ra/{\cal B}_{a_1} = \la x\ra/{\cal B}_{a_1}$.  As $P$ is semiregular and contained in the center of $\la x,y\ra$, we have $P$ is contained in $\la y\ra$, and so $\la y\ra = \la x\ra$.  We have shown that $\la x\ra$ and $\la y\ra$ are conjugate in $\la x,y\ra$.  That is, that $\la x\ra$ is a pronormal subgroup of $\la x,y\ra$.

The ``consequently" part of the result now follows easily.  Let $G\le S_n$ be $9/8$-closed, $R$ a regular cyclic subgroup of $G$, and $g\in G$.  Set $R = \la x\ra$, and $y = g^{-1}xg$.  If $G$ is $9/8$-closed, then for every conjugate $y'$ of $y$ in $\la x,y\ra$ and normal block system ${\cal B}$ of $\la x,y'\ra$, the ${\cal B}$ fixing block system of $\la x,y'\ra$ is $\{\Z_n\}$.  By the first part of this result, we see that $\la y\ra$ is conjugate to $\la x\ra$ in $\la x,y\ra$.  Thus $R = \la x\ra$ is pronormal in $G$.
\end{proof}

Combining the previous result with Lemma \ref{CI} we obtain the following result.

\begin{cor}\label{set system CI}
Let $T$ be a connected circulant colored tuple system of order $n$ with corresponding set system $S$ such that $\Aut(S)$ is $9/8$-closed.  Then $T$ is a CI-colored tuple system of $\Z_n$.
\end{cor}

There is a slight and subtle gap in the proof of \cite[Theorem 1.1]{KoikeKMM2019}.  Namely, in $\S 2$, they discuss how the isomorphism problem for configurations reduces to the connected case, as it is clear that if one has an isomorphism between connected components of two isomorphic disconnected set systems, then one can write down an isomorphism between the two set systems.  However, if the set system is a Cayley set system of $G$, then the set of components will be the set of left cosets of some subgroup $H\le G$ \cite[Example 2.3.8]{Book}.  If $H$ is a CI-group with respect to the particular set system one is considering, such as configurations, in order to show $G$ is a CI-group as well, automorphisms of $H$ must extend to automorphisms of $G$.  In general, though, automorphisms of some subgroup $H\le G$ need not extend to automorphisms of $G$.  The final result of this section generalizes \cite[Theorem 1.1]{KoikeKMM2019} (as a symmetric configuration is a partial Sylvester-Gallai design) and fills the small gap in their proof.  We need one more natural definition.

\begin{definition}
Let $T$ be a disconnected colored tuple system.  A {\bf component of $T$} is the set of all tuples in $T$ all of whose coordinates are contained in the vertex set of a component of the underlying set system of $T$.
\end{definition}

Note that a partial Sylvester-Gallai design that has a line must have at least $3$ points.

\begin{cor}
Let $n\ge 3$ be an integer.  Then $\Z_n$ is a CI-group with respect to colored tuple systems whose underlying set systems are partial Sylvester-Gallai designs.
\end{cor}

\begin{proof}
Let $T$ be a circulant colored tuple system whose underlying set system $D$ is a partial Sylvester-Gallai design.  First suppose $D$ is connected.  By Theorem \ref{9/8 object} we have that $\Aut(T)\le\Aut(D)$ is $9/8$-closed.  By Corollary \ref{set system CI} we see that $T$ (and $D$) are CI-objects of $\Z_n$.  %Hence $\Z_n$ is a CI-group with respect to colored tuple systems whose underlying set systems are partial Sylvester-Gallai designs.

If $D$ is a disconnected partial Sylvester-Gallai design, then the component $S$ of $T$ which contains $0$ is a colored tuple system whose underlying set system $E$ is a partial Sylvester-Gallai design of order $m = \vert V(S)\vert$, for some $m\vert n$.  It is also circulant, as the components of $T$ are clearly blocks of $\Aut(T)$, and a block system of a group that contains a regular abelian group is normal by \cite[Theorem 2.2.19]{Book}.  Additionally, the set of points of $S$ is a (cyclic) subgroup of $\Z_n$.  By the first part of this proof, $S$ is a CI-object of $\Z_m$.  Let $T'$ be a colored circulant tuple system whose underlying set system is a partial Sylvester-Gallai design $D'$ of order $n$ isomorphic to $T$.  Then $T'$ is disconnected, and letting $S'$ be the connected component of $T'$ that contains $0$, we see that $S$ and $S'$ are isomorphic by a group automorphism of the unique subgroup $H$ of $\Z_n$ of order $m$.  By  \cite{Li1999}, every automorphism of $H$ is the restriction of a group automorphism of $\Z_n$ to $H$, and so there exists a group automorphism $\alpha$ of $\Z_n$ such that $\alpha(S) = S'$.  As the points of the connected components of both $T$ and $T'$ are the left cosets of $H$ in $\Z_n$, and $\alpha$ maps cosets of $H$ to cosets of $H$, we see $\alpha(T) = T'$.  Thus every disconnected colored tuple system whose underlying set system is a partial Sylvester-Gallai design is also a CI-object of $\Z_n$, and $\Z_n$ is a CI-group with respect to colored tuple systems whose underlying set system is a partial Sylvester-Gallai design.
\end{proof}

The authors of \cite{KoikeKMM2019} ended their paper with the natural question of determining which groups $G$ are CI-groups with respect to Cayley symmetric configurations of $G$.  We end this section with what we believe are equally obvious questions.

\begin{problem}
Which groups $G$ have the property that $G$ is pronormal in every $9/8$-closed group that contains $G_L$?
\end{problem}

By Theorem \ref{9/8 object}, the next problem is a special case of the previous problem.

\begin{problem}
Which groups $G$ have the property that $G$ is a CI-group with respect to colored tuple systems whose underlying set systems are Cayley partial Sylvester-Gallai designs of $G$?
\end{problem}

\section{Unit circulant digraphs and $3/2$-closed groups}

\begin{definition}
Let $n$ be a positive integer and $S\subseteq\Z_n^*$, the set of units in $\Z_n$.  A {\bf unit circulant digraph} of order $n$ is a Cayley digraph $\Cay(\Z_n,S)$.
\end{definition}

Toida \cite{Toida1977} conjectured that every unit circulant digraph is a CI-digraph.  This conjecture was independently confirmed by Klin, Muzychuk and P\"oschel  \cite{MuzychukKP2001} (using Schur rings), as well as by the author and Joy Morris \cite{DobsonM2002} (using group theoretic techniques).  Our goal in this section is to generalize the fact that Toida's conjecture is true, as well as generalize the notion of a unit circulant digraph to groups other than cyclic groups.  It is also very important to note that the proofs given here are much shorter and more complete than those given in \cite{MuzychukKP2001} or \cite{DobsonM2002}.  After some terms that are needed are defined, we begin with the relationship between the automorphism group of a unit circulant digraph and $3/2$- and $9/8$-closed groups.  We next give several definitions which we will need concerning digraphs.

\begin{definition}
Let $\Gamma_1$ and $\Gamma_2$ be digraphs.  The {\bf wreath product of $\Gamma_1$ and $\Gamma_2$}, denoted $\Gamma_1\wr\Gamma_2$, is the digraph with vertex set $V(\Gamma_1)\times V(\Gamma_2)$ and arcs $((u,v),(u,v'))$ for $u\in V(\Gamma_1)$ and $(v,v')\in A(\Gamma_2)$ or $((u,v),(u',v'))$ where $(u,u')\in A(\Gamma_1)$ and $v,v'\in V(\Gamma_2)$.
\end{definition}

The wreath product of two graphs was introduced by Harary \cite{Harary1959}, who wanted a graph product whose automorphism group was the wreath product of the automorphism groups of its factor graphs.  It is easy to show that $\Aut(\Gamma_1)\wr\Aut(\Gamma_2)\le\Aut(\Gamma_1\wr\Gamma_2)$.

\begin{definition}
Let $\Gamma$ be a vertex-transitive digraph whose automorphism group contains a transitive subgroup $G$ with a block system ${\cal B}$.  Define the {\bf block quotient digraph of $\Gamma$ with respect to ${\cal B}$}, denoted $\Gamma/{\cal B}$, to be the digraph with vertex set ${\cal B}$ and arc set $\{(B,B'):B,B'\in{\cal B}, B\neq B',{\rm\ and\ }(u,v)\in A(\Gamma){\rm\ for\ some\ }u\in B{\rm\ and\ } v \in B'\}$.
\end{definition}

See \cite[\S 4.2]{Book} and \cite[\S 2.6]{Book} for more information and examples on the digraph wreath product and block quotient digraphs, respectively.

\begin{definition}
Let $\Gamma$ be a digraph.  Define an equivalence relation $R$ on $V(\Gamma)$ by $u\ R\ v$ if and only if the out- and in-neighbors of $u$ and $v$ are the same.  Then $R$ is an equivalence relation on $V(\Gamma)$. We say $\Gamma$ is {\bf irreducible} if the equivalence classes of $R$ are singletons, and {\bf reducible} otherwise.
\end{definition}

{\color{purple}

The equivalence relation above was introduced for graphs by Sabidussi \cite[Definition 3]{Sabidussi1959}, and independently rediscovered by Kotlov and Lov\'asz \cite{KotlovL1996}, who call $u$ and $v$ {\bf twins}, and Wilson \cite{Wilson2003}, who calls reducible graphs {\bf unworthy}.  Sabidussi observed in \cite{Sabidussi1964} that $R$ is a $G$-congruence for $G\le\Aut(\Gamma)$ and graphs $\Gamma$.  This is true for digraphs as well, and hence the set of equivalence classes of $R$ is a block system of $G$ by \cite[Theorem 3.2.2]{Book}.  It is easy to show, using \cite[Theorem 4.2.15]{Book}, that a vertex-transitive digraph $\Gamma$ is reducible if and only if it can be written as a wreath product $\Gamma_1\wr\bar{K}_n$ for some positive integer $n\ge 2$, where $\bar{K}_n$ is the complement of the complete graph.

}

\begin{thrm}\label{unit circ auto}
The automorphism group of every unit circulant digraph $\Gamma$ of order $n$ is $3/2$-closed.  Additionally, either $\Aut(\Gamma)$ is $9/8$-closed, or $\Gamma$ is reducible, $\Aut(\Gamma) = K\wr S_m$ for some $m\ge 2$, and $K\le S_{n/m}$ is $9/8$-closed.
\end{thrm}

\begin{proof}
Let $\Gamma = \Cay(\Z_n,S)$ for some $S\subseteq\Z_n^*$ be a unit circulant digraph.  We will first show that $\Aut(\Gamma)$ is $3/2$-closed.  Let $G = \Aut(\Gamma)$ and $H\le G$ transitive with normal block system ${\cal B}$.  If $\WStab_H(B) = 1$ there is nothing to prove.  So we assume $\WStab_H(B)\not = 1$ for some $B\in{\cal B}$, and let ${\cal E}$ be the ${\cal B}$-fixing block system of $H$.

Let $E_0, E\in {\cal E}$ where $0\in E_0$.  Then $\Gamma[E]$ is not an empty graph if and only if $\Gamma[E_0]$ is not an empty graph.  Also, $E_0\le\Z_n$ and as $\Gamma$ is a unit circulant digraph, we see that $E_0\cap S = \emptyset$.  Then in $E_0$, between any two distinct blocks of ${\cal B}$ contained within $E_0$, there are no arcs.  Similarly, between any two distinct blocks of ${\cal B}$ contained within $E$, there are no arcs.

Now let $B,B'\in{\cal B}$ such that $B$ and $B'$ are contained in different blocks $E,E'\in{\cal E}$.  As $B\not\equiv B'$,  $\WStab_G(B)^{B'}$ is transitive.  Hence if there is an arc $(x,x')\in A(\Gamma)$ with $x\in B$ and $x'\in B'$, then $(x,x')\in A(\Gamma)$ for every $x'\in B'$.  As ${\cal B}$ is normal, this implies that $(x,x')\in A(\Gamma)$ for every $x\in B$ and $x'\in B'$.

We conclude by \cite[Theorem 4.2.15]{Book} that $\Gamma\cong \Gamma/{\cal B}\wr\Gamma[B] = \Gamma/{\cal B}\wr\bar{K}_B$.  Thus $G$ contains $\Aut(\Gamma/{\cal B})\wr S_B$ and letting $L\le\Aut(\Gamma/{\cal B})\wr S_B$ be the largest subgroup of $G$ that has ${\cal B}$ as a block system, we see $\fix_L(B)^B = S_B$ for every $B\in{\cal B}$.  Thus $\Aut(\Gamma)$ is $3/2$-closed by definition.

The above argument also shows that if ${\cal E}\not = \{\Z_n\}$, then $\Gamma$ is reducible.  If ${\cal E} = \{\Z_n\}$ for every $H\le G$ that is transitive and block system ${\cal B}$ of $H$, then $G$ is $9/8$-closed.  It only remains to show that if $\Gamma$ is reducible, then $\Aut(\Gamma) = K\wr S_m$ for some $m\ge 2$ and $K\le S_{n/m}$ is $9/8$-closed.

Let $m\le n$ be maximum such that $\Gamma = \Gamma_1\wr\bar{K}_m$, where $\Gamma_1$ is a circulant digraph of order $n/m$.  If $m = n$, then $\Gamma = \bar{K}_n$ (which is a unit circulant digraph).  The result then follows with $K = 1$.  Otherwise, $\Gamma_1$ has at least $2$ vertices, and $\Gamma_1$ is irreducible by \cite[Lemma 1 (ii) and (iii)]{Sabidussi1964} and their digraph analogues.  Let ${\cal B}$ be the set of cosets of the unique subgroup $M$ of $\Z_n$ of order $m$.  Then $\Gamma/{\cal B} = \Gamma_1$.  Let ${\cal C}/{\cal B}$ be any block system of $K = \Aut(\Gamma_1)$.  Then ${\cal B}\preceq{\cal C}$, and so by \cite[Lemma 2.2]{Dobson2023apreprint}, we see that if ${\cal E'}/{\cal B}$ is the ${\cal C}/{\cal B}$-fixing block system of $\Aut(\Gamma_1)$, then ${\cal E}'$ is the ${\cal C}$-fixing block system of $G$.  As $m$ is maximum and $\Aut(\Gamma)$ is $3/2$-closed, we see that ${\cal E}'/{\cal B} = \{\Z_n/M\}$, and $K$ is $9/8$-closed.
\end{proof}

Our next goal is to give in Theorem \ref{3/2 structure} the relationship between $9/8$-closed groups and $3/2$-closed groups.  We begin with an additional term and preliminary results.

\begin{definition}
Let $g\in S_n$.  We define the {\bf support} of $g$ to be all $x\in\Z_n$ such that $g(x)\not = x$.  Let $G\le S_n$.  The {\bf support} of $G$ is the union of the support of all elements of $G$.
\end{definition}

\begin{lem}\label{maximal symmetric}
Let $G\le S_n$ be transitive and contain a subgroup $H$ with the property that there is $T\subset\Z_n$ such that $H^T = S_T$ and $H^{\Z_n\setminus T} = 1$.  Then there exists a block system ${\cal B}$ of $G$ such that $G = (G/{\cal B})\wr S_m$, for some $m\vert n$.
\end{lem}

\begin{proof}
We first establish the following claim: If there is $K\le G$ and $U\subset\Z_n$ such that $K^U = S_U$, $K^{\Z_n\setminus U} = 1$, and $T\cap U\not = \emptyset$ then $\la H,K\ra^{T\cup U} = S_{T\cup U}$ and $\la H,K\ra^{\Z_n\setminus (T\cup U)} = 1$.  The last condition should be clear as the support of both $H$ and $K$ is contained in $T\cup U$.  In order to show that $\la H,K\ra = S_{T\cup U}$, it suffices to show that $\la H\cup K\ra$ contains every transposition on $T\cup U$.  Let $(a,b)$ be a transposition on $T\cup U$.  If either $a,b\in T$ or $a,b\in U$, then clearly $(a,b)\in \la H,K\ra$.  Otherwise, suppose, say, $a\in T$ and $b\in U$ but $a\not\in U$ and $b\not \in T$, with the other case being analogous.  As $T\cap U\not = \emptyset$, there is $c\in T\cap U$.  Then $(a,c)\in H$, $(c,b)\in K$, and $(a,c)(c,b)(a,c) = (a,b)\in \la H,K\ra$.  This establishes the claim.

We next assume that $H$ is chosen so that $\vert T\vert$ is as large as possible.  Let $g\in G$.  If $g(T)\cap T\not = \emptyset$, then by the claim $\la H,gHg^{-1}\ra$ has the property that its induced action on $T\cup g(T)$ is the symmetric group on $T\cup g(T)$, while its induced action on $\Z_n\setminus (T\cup g(T))$ is trivial.  As $m$ was chosen so that $\vert T\vert$ is as large as possible, it must be that $g(T) = T$.  Hence $T$ is a block of $G$.  We then let ${\cal B}$ be the set of all blocks conjugate to $T$, and the result follows.
\end{proof}

\begin{thrm}\label{quotient 9/8 closed}
Let $G\le S_n$ be a $3/2$-closed transitive group with a normal block system ${\cal B}$ with ${\cal E}$ the ${\cal B}$-fixer block system of $G$.  If ${\cal E}\not = \{\Z_n\}$, then $G/{\cal B}$ is $9/8$-closed.
\end{thrm}

\begin{proof}
Let $H\le G$ be transitive such that $H/{\cal B}$ has a normal block system ${\cal C}/{\cal B}$ with ${\cal E}'/{\cal B}$ the ${\cal C}/{\cal B}$-fixer block system of $H/{\cal B}$.  By hypothesis, ${\cal E} = {\cal B}\preceq{\cal E'}$.  By \cite[Lemma 2.2]{Dobson2023apreprint}, the ${\cal C}$-fixer block system of $H$ is ${\cal E}'$.  As $G$ is $3/2$-closed, we see that either ${\cal E}' = \{\Z_n\}$ and the result follows, or ${\cal E}' = {\cal C}$ and that $\fix_H({\cal C})^C = S_C$ for every $C\in{\cal C}$.  However, if $\fix_H({\cal C})^C = S_C$, $H$ cannot have ${\cal B}$ as a block system as ${\cal B}\prec{\cal C}$.  But ${\cal B}$ is a block system of $G\ge H$, a contradiction.
\end{proof}

\begin{definition}
Let $X$ and $Y$ be sets, and $G\le S_X$ and $H\le S_Y$ be transitive groups.  Then ${\cal B} = \{\{(x,y):y\in Y\}:x\in X\}$ is a normal block system of $G\wr H$, called the {\bf lexi-partition of $G\wr H$ corresponding to $Y$}.
\end{definition}

\begin{thrm}\label{3/2 structure}
The class of $3/2$-closed groups is the disjoint union of the class of $9/8$-closed groups together with the class of groups obtained from the $9/8$-closed groups by wreathing them with a symmetric group.  That is, let ${\cal K}$ be the class of $9/8$-closed groups.  Then the class of $3/2$-closed groups is the union of ${\cal K}$ together with $\{K\wr S_n:K\in{\cal K}\ and\ n\ge 2\ is\ an\ integer\}$.
\end{thrm}

\begin{proof}
It is clear that $9/8$-closed groups are $3/2$-closed.  Let ${\cal L} = \{K\wr S_n:K\in{\cal K}{\rm \ and\ }n\ge 2{\rm \ is\ an\ integer}\}$.  It is also clear that ${\cal K}\cap{\cal L} = \emptyset$.  We first show that if $G\in{\cal L}$, then $G$ is $3/2$-closed.

Let $G\in {\cal L}$.  Write $G = K\wr S_n$, where $K\in{\cal K}$.  Let $H\le G$ be transitive with a normal block system ${\cal B}$ such that the ${\cal B}$-fixing block system ${\cal E}$ of $H$ has more than one block.  Let ${\cal C}$ be the lexi-partition of $G$ with respect to $S_n$.  If ${\cal B} = {\cal C}$, then we have the largest subgroup of $G$ with ${\cal B}$ a block system is simply $G$.  As $G = K\wr S_n$, we have that the ${\cal B}$-fixing block system of $G$ is ${\cal B}$, and $\fix_G({\cal B})^B = S_n$.  So we need only consider when ${\cal B}\not = {\cal C}$.

We proceed by contradiction, and so assume that $G = K\wr S_n$ is not $3/2$-closed.   We may assume without loss of generality that $G$ is chosen so that $G = K\wr S_n$ with $n\ge 2$ as small as possible such that $K\wr S_n$ is not $3/2$-closed.  As ${\cal B}$ and ${\cal C}$ are block systems of $G$, there either exist $B\in{\cal B}$ and $C\in{\cal C}$ such that $\vert B\cap C\vert\ge 2$, or $\vert B\cap C\vert\le 1$ for every $B\in{\cal B}$ and $C\in{\cal C}$.  We consider the latter case first, and then reduce the former case to the latter case.

Suppose $\vert B\cap C\vert \le 1$ for every $B\in{\cal B}$ and $C\in{\cal C}$.  Then $H/{\cal C}\le K$ and $\fix_H({\cal B})/{\cal C}\tl H/{\cal C}$.  Additionally, $H/{\cal C}\cong H/\fix_H({\cal C})$, and as $\vert B\cap C\vert\le 1$ for every $B\in{\cal B}$ and $C\in{\cal C}$, we see that $\fix_H({\cal B})\cap\fix_H({\cal C}) = 1$.  Thus $\WStab_H(B)\cap\fix_H({\cal C}) = 1$ for every $B\in{\cal B}$.  Let ${\cal B}'$ be the set of orbits of $\fix_H({\cal B})/{\cal C}$.  As $\fix_H({\cal B})\cap\fix_H({\cal C}) = 1$, ${\cal B}'$ is a normal block system of $K$ whose blocks are not singletons.  As the ${\cal B}$-fixing block system ${\cal E}$ of $H$ has more than one block, ${\cal B}$ has more than one block and so ${\cal B}'$ has more than one block.  Thus ${\cal B}'$ is not trivial.  Let $C\in{\cal C}$ such that $\vert C\cap B\vert = 1$.  As $\WStab_H(B)$ fixes $B$ pointwise, it fixes the point in $C\cap B$, and so fixes $C$.  Thus $\WStab_H(B)$ fixes each block of $C$ with a point in $B$.  Hence $(\WStab_H(B)/{\cal C})^{B/{\cal C}} = 1$.  Let $B'\in{\cal B}$ such that  $\WStab_H(B)$ is transitive on $B'$.  Let $C',C''\in {\cal C}$ such that $\vert B'\cap C'\vert = \vert B'\cap C''\vert = 1$, with $\{x'\} = B'\cap C'$ and $\{x''\} = B'\cap C''$.  Then there exists $g\in \WStab_H(B)$ such that $g(x') = x''$.  Then $g/{\cal C}(C') = C''$ and so $\WStab_H(B)/{\cal C}$ is transitive on $B'/{\cal C}$.  This implies that the ${\cal B}'$-fixing block system of $H/{\cal C}\le K$ has at least two blocks, contradicting the assumption that $K$ is $9/8$-closed.  Thus we may assume that $\vert B\cap C\vert\ge 2$ for some $B\in{\cal B}$ and $C\in{\cal C}$.

Suppose $\vert B\cap C\vert\ge 2$ for some $B\in{\cal B}$ and $C\in{\cal C}$.  Then $B\cap C$ is a block of $H$ with ${\cal D}$ the block system of $H$ that contains $B\cap C$.  As $G = K\wr S_n$ and ${\cal D}\prec{\cal C}$, we see that $G$ contains a subgroup permutation isomorphic to $H/{\cal D}\wr S_D$.  By the minimality of $n$ and that $n\ge 2$, we see that $K\wr S_{n-\vert D\vert}$ is $3/2$-closed.  As ${\cal B}\not = {\cal C}$, repeating the previous argument, if necessary, we see there is $m\le n$ and block system ${\cal F}\prec{\cal B}$ (and ${\cal F}\prec{\cal C}$) such that $K\wr S_{n-m}$ is $3/2$-closed and $\vert (B/{\cal F})\cap (C/{\cal F})\vert\le 1$ for every $B\in{\cal B}$ and $C\in{\cal C}$.  By arguments in the preceding paragraph we see that $K$ is not $9/8$-closed, a contradiction.  Hence $K\wr S_n$ is $3/2$-closed.

Let $G$ be $3/2$-closed.  As a $9/8$-closed groups is $3/2$-closed, we may assume that $G$ is not $9/8$-closed.  Then there exists a transitive subgroup $H\le G$ with a normal block system ${\cal B}$ for which the ${\cal B}$-fixer block system ${\cal E}$ is not $\{\Z_n\}$.  As $G$ is $3/2$-closed, we see that ${\cal E} = {\cal B}$ and $\fix_H({\cal B})^B = S_B$ for every $B\in{\cal B}$.  We may then write $H = (H/{\cal B})\wr S_k$, where $k = \vert B\vert$, $B\in{\cal B}$.  By Lemma \ref{maximal symmetric}, there is a block system ${\cal C}$ of $G$ such that $G\cong G/{\cal C}\wr S_m$ for some $m\vert n$.  The result now follows by Lemma \ref{quotient 9/8 closed}.
\end{proof}

\begin{definition}
Let $G\le S_n$ be a transitive group.  By the $3/2$-closure of $G$, denoted $G^{(3/2)}$, we mean the intersection of all $3/2$-closed groups that contain $G$.
\end{definition}

It is easy to see that the $3/2$-closure of $G$ is $3/2$-closed.

\begin{lem}\label{3/2 closure}
Let $G\le S_n$ be a transitive group.  Then $G^{(3/2)}$ is $3/2$-closed.
\end{lem}

\begin{proof}
Let $H\le G^{(3/2)}$ be transitive with a block system ${\cal B}$.  Let $L$ be a $3/2$-closed group that contains $G$ and $K\le L$ be the largest subgroup of $L$ that has ${\cal B}$ as a block system.  Then $H\le K$.  Also, if $\WStab_H({\cal B})\not = 1$, then the ${\cal B}$-fixing block system of $H$ is ${\cal B}$.  As $H\le K\le L$, the ${\cal B}$-fixing block system of $K$ is ${\cal B}$, and $\fix_K({\cal B})^B = S_B$ for every $B\in{\cal B}$.  Let $M$ be the largest subgroup of $G^{(3/2)}$ that has ${\cal B}$ as a block system.  As $L$ is arbitrary, the ${\cal B}$-fixing block system of $M$ is ${\cal B}$ and $\fix_M({\cal B})^B = S_B$ for every $B\in{\cal B}$.  Thus $G^{(3/2)}$ is $3/2$-closed.
\end{proof}

\begin{thrm}\label{Toida generalization}
Let $x,y\in S_n$ such that $\la x\ra$ and $\la y\ra$ are regular cyclic subgroups.  Then $\la x\ra$ and $\la y\ra$ are conjugate in $G = \la x,y\ra^{(3/2)}$.  Consequently, every CI-object of $\Z_n$ with $3/2$-closed automorphism group is a CI-object of $\Z_n$.  In particular, every unit circulant digraph of order $n$ is a CI-digraph of $\Z_n$.
\end{thrm}

\begin{proof}
By Theorem \ref{configuration result}, we may assume that $G$ is not $9/8$-closed, in which case by Theorem \ref{3/2 structure} we see that $G = K\wr S_m$ for some $9/8$-closed group $K$ and $m\vert n$.  Let ${\cal B}$ be the block system of $G$ with blocks of size $m$.  By Theorem \ref{configuration result} there is $\delta\in G$ such that $\la\delta^{-1}y\delta\ra/{\cal B} = \la x\ra/{\cal B}$.  We may thus assume without loss of generality that $\la y\ra/{\cal B} = \la x\ra/{\cal B}$.  As for every $a\in\Z_n^*$ we have $\la y\ra$ is conjugate to $\la x\ra$ if and only if $\la y\ra$ is conjugate to $\la x^a\ra = \la x\ra$, we may assume without loss of generality that $y/{\cal B} = x/{\cal B}$.  We may then assume that $\delta/{\cal B}\in\la x\ra/{\cal B}$ as a regular abelian group is self-centralizing \cite[Corollary 2.2.18]{Book}, in which case we may assume without loss of generality that $\delta/{\cal B} = 1$.  As the ${\cal B}$-fixing block system of $H$ is ${\cal E}$ and $\fix_K({\cal B})^B = S_B$, we have that $\delta\in H$, and so $\la x\ra$ and $\la y\ra$ are conjugate in $G$.

Now let $X$ be a Cayley object of $\Z_n$ in some class ${\cal K}$ of combinatorial objects such that $\Aut(X)$ has a $3/2$-closed automorphism group.  By the first part of this result, if $\phi\in S_n$ such that $\phi^{-1}(\Z_n)_L\phi\le\Aut(X)$, then there exists $\delta\in\Aut(X)$ such that $\delta^{-1}\phi^{-1}(\Z_n)_L\phi\delta = (\Z_n)_L$.  Hence by Lemma \ref{CI}, we have that $X$ is a CI-object of $\Z_n$.  In particular, by Theorem \ref{unit circ auto} the automorphism group of a unit circulant digraph is $3/2$-closed.  Thus every unit circulant digraph is a CI-digraph of $\Z_n$ and the result follows.
\end{proof}

With the idea of $3/2$-closed groups, we can define ``unit" Cayley digraphs of groups other than the cyclic group, as well as double coset digraphs.

\begin{definition}
Let $G$ be a group, $H\le G$, and $S\subset G$ such that $HSH = S$.  We say that the double coset digraph $\Cos(G,H,S)$ is a {\bf unit coset digraph of $G$} if $\Aut(\Cos(G,S))$ is $3/2$-closed.
\end{definition}

\begin{problem}\label{unit CI}
For which groups $G$ is it true that every unit Cayley digraph of $G$ is a CI-digraph of $G$?
\end{problem}

One may ask why, in the previous problem, we did not generalize the notion of a CI-digraph to double coset digraphs, and pose Problem \ref{unit CI} in that context.  The reason for this is that it was shown in \cite{BarberDpreprint} that the isomorphism problem for coset digraphs of $G$ is equivalent to the isomorphism problem for Cayley digraphs of $G$.  There is thus no need for the more general form.

\providecommand{\bysame}{\leavevmode\hbox to3em{\hrulefill}\thinspace}
\providecommand{\MR}{\relax\ifhmode\unskip\space\fi MR }
% \MRhref is called by the amsart/book/proc definition of \MR.
\providecommand{\MRhref}[2]{%
  \href{http://www.ams.org/mathscinet-getitem?mr=#1}{#2}
}
\providecommand{\href}[2]{#2}

%\bibliography{References}{}

\begin{thebibliography}{10}

\bibitem{Adam1967}
A.~\'Ad\'am, \emph{Research problem 2-10}, J. Combin. Theory \textbf{2} (1967),
  393.

\bibitem{Babai1977}
L.~Babai, \emph{Isomorphism problem for a class of point-symmetric structures},
  Acta Math. Acad. Sci. Hungar. \textbf{29} (1977), no.~3-4, 329--336.
  \MR{MR0485447 (58 \#5281)}

\bibitem{BarberDpreprint}
Rachel Barber and Ted Dobson, \emph{Finding automorphism groups of double coset
  graphs and cayley graphs are equivalent}, Preprint, 2021.

\bibitem{DixonM1996}
John~D. Dixon and Brian Mortimer, \emph{Permutation groups}, Graduate Texts in
  Mathematics, vol. 163, Springer-Verlag, New York, 1996. \MR{MR1409812
  (98m:20003)}

\bibitem{Dobson2008a}
Edward Dobson, \emph{On isomorphisms of circulant digraphs of bounded degree},
  Discrete Math. \textbf{308} (2008), no.~24, 6047--6055. \MR{MR2464896}

\bibitem{DobsonM2002}
Edward Dobson and Joy Morris, \emph{Toida's conjecture is true}, Electron. J.
  Combin. \textbf{9} (2002), no.~1, Research Paper 35, 14 pp. (electronic).
  \MR{MR1928787 (2003h:05096)}

\bibitem{Dobson2023epreprint}
Ted Dobson, \emph{Automorphism groups of circulant digraphs of odd prime power
  order, in preparation.}

\bibitem{Dobson2023apreprint}
\bysame, \emph{Towards inductive proofs in algebraic combinatorics}.

\bibitem{Book}
Ted Dobson, Aleksander Malni\v{c}, and Dragan Maru\v{s}i\v{c}, \emph{Symmetry
  in graphs}, Cambridge Studies in Advanced Mathematics, vol. 198, Cambridge
  University Press, Cambridge, 2022. \MR{4404766}

\bibitem{DobsonMS2022}
Ted Dobson, Mikhail Muzychuk, and Pablo Spiga, \emph{Generalised dihedral
  {CI}-groups}, Ars Math. Contemp. \textbf{22} (2022), no.~2, Paper No. 7, 18.
  \MR{4449170}

\bibitem{ElspasT1970}
Bernard Elspas and James Turner, \emph{Graphs with circulant adjacency
  matrices}, J. Combinatorial Theory \textbf{9} (1970), 297--307. \MR{MR0272659
  (42 \#7540)}

\bibitem{Grunbaum2009}
Branko Gr{\"u}nbaum, \emph{Configurations of points and lines}, Graduate
  Studies in Mathematics, vol. 103, American Mathematical Society, Providence,
  RI, 2009. \MR{2510707}

\bibitem{Harary1959}
Frank Harary, \emph{On the group of the composition of two graphs}, Duke Math.
  J \textbf{26} (1959), 29--34. \MR{0110648 (22 \#1523)}

\bibitem{Hungerford1980}
Thomas~W. Hungerford, \emph{Algebra}, Graduate Texts in Mathematics, vol.~73,
  Springer-Verlag, New York, 1980, Reprint of the 1974 original. \MR{MR600654
  (82a:00006)}

\bibitem{KellyN1973}
L.~M. Kelly and S.~Nwankpa, \emph{Affine embeddings of {S}ylvester-{G}allai
  designs}, J. Combinatorial Theory Ser. A \textbf{14} (1973), 422--438.
  \MR{314656}

\bibitem{KoikeKMM2019}
Hiroki Koike, Istv\'{a}n Kov\'{a}cs, Dragan Maru\v{s}i\v{c}, and Mikhail
  Muzychuk, \emph{Cyclic groups are {CI}-groups for balanced configurations},
  Des. Codes Cryptogr. \textbf{87} (2019), no.~6, 1227--1235. \MR{3947344}

\bibitem{KotlovL1996}
Andrew Kotlov and L\'{a}szl\'{o} Lov\'{a}sz, \emph{The rank and size of
  graphs}, J. Graph Theory \textbf{23} (1996), no.~2, 185--189. \MR{1408346}

\bibitem{Li1999}
Cai~Heng Li, \emph{A complete classification of finite homogeneous groups},
  Bull. Austral. Math. Soc. \textbf{60} (1999), no.~2, 331--334. \MR{1715229}

\bibitem{Li2002}
\bysame, \emph{On isomorphisms of finite {C}ayley graphs---a survey}, Discrete
  Math. \textbf{256} (2002), no.~1-2, 301--334. \MR{MR1927074 (2003i:05067)}

\bibitem{Muzychuk1999}
Mikhail Muzychuk, \emph{On the isomorphism problem for cyclic combinatorial
  objects}, Discrete Math. \textbf{197/198} (1999), 589--606, 16th British
  Combinatorial Conference (London, 1997). \MR{MR1674890 (2000e:05165)}

\bibitem{MuzychukKP2001}
Mikhail Muzychuk, Mikhail Klin, and Reinhard P{\"o}schel, \emph{The isomorphism
  problem for circulant graphs via {S}chur ring theory}, Codes and association
  schemes ({P}iscataway, {NJ}, 1999), DIMACS Ser. Discrete Math. Theoret.
  Comput. Sci., vol.~56, Amer. Math. Soc., Providence, RI, 2001, pp.~241--264.
  \MR{MR1816402 (2002g:05128)}

\bibitem{PisanskiS2013}
Toma\v{z} Pisanski and Brigitte Servatius, \emph{Configurations from a
  graphical viewpoint}, Birkh\"{a}user Advanced Texts: Basler Lehrb\"{u}cher.
  [Birkh\"{a}user Advanced Texts: Basel Textbooks], Birkh\"{a}user/Springer,
  New York, 2013. \MR{2978043}

\bibitem{Sabidussi1958}
Gert Sabidussi, \emph{On a class of fixed-point-free graphs}, Proc. Amer. Math.
  Soc. \textbf{9} (1958), 800--804. \MR{MR0097068 (20 \#3548)}

\bibitem{Sabidussi1959}
\bysame, \emph{The composition of graphs}, Duke Math. J \textbf{26} (1959),
  693--696. \MR{MR0110649 (22 \#1524)}

\bibitem{Sabidussi1964}
\bysame, \emph{Vertex-transitive graphs}, Monatsh. Math. \textbf{68} (1964),
  426--438. \MR{MR0175815 (31 \#91)}

\bibitem{Toida1977}
Shunichi Toida, \emph{A note on \'{A}d\'am's conjecture}, J. Combinatorial
  Theory Ser. B \textbf{23} (1977), no.~2-3, 239--246. \MR{0463014 (57 \#2978)}

\bibitem{Wielandt1969}
H.~Wielandt, \emph{Permutation groups through invariant relations and invariant
  functions}, lectures given at The Ohio State University, Columbus, Ohio,
  1969.

\bibitem{Wilson2003}
Steve Wilson, \emph{A worthy family of semisymmetric graphs}, Discrete Math.
  \textbf{271} (2003), no.~1-3, 283--294. \MR{1999550}

\end{thebibliography}
%\bibliographystyle{amsplain}

\end{document}